\documentclass[10pt]{amsart}
\usepackage{graphics}
\usepackage[english]{babel}
\usepackage{amssymb,amsmath,amsfonts}
\usepackage{datetime}
\usepackage{color}
\usepackage{units}
\usepackage{ulem}
\usepackage{lipsum}
\usepackage{amsfonts}
\usepackage{enumerate}
\usepackage{cite}
\usepackage{tikz}
\usetikzlibrary{matrix}
\usepackage[pdftex]{hyperref}
\RequirePackage{amscd}
\RequirePackage{epic}
\RequirePackage[all]{xy}
\RequirePackage{url}
\RequirePackage[shortlabels]{enumitem}
\usepackage{empheq}
\usepackage{epsfig}
\usepackage{lineno} 
 \usepackage{perpage}
 \MakePerPage{linenumbers}
\usepackage[english]{babel}

\usepackage[utf8]{inputenc}
\usepackage{cite}
\RequirePackage{amscd}
\RequirePackage{epic}
\RequirePackage{eepic}
\RequirePackage[all]{xy}
\RequirePackage{url}
\RequirePackage[shortlabels]{enumitem}
\usepackage{empheq}
\usepackage{nomencl}
\usepackage{epsfig}
\usepackage{graphicx}

\newcommand{\comment}[1]{}
   
    \newcommand{\set}[1]{{\left\{#1\right\}}}
\newcommand{\pa}[1]{{\left(#1\right)}}
\newcommand{\sq}[1]{{\left[#1\right]}}

\newcommand{\abs}[1]{{\left|#1\right|}}
\newcommand{\norm}[1]{{\left \|#1\right \|}}

\newcommand{\T}{\mathbb{T}}
\newcommand{\Z}{\mathbb{Z}}
\newcommand{\R}{\mathbb{R}}
\newcommand{\C}{\mathbb{C}}

\newcommand{\dg}{{\mathtt{D}_{\g,\cS}}}

\newcommand{\eps}{\varepsilon}

\newcommand{\pan}{{\mathcal Q}}
\newcommand{\na}{\widehat{n}}

\newcommand{\co}[1]{\textit{#1}}
\newcommand{\gr}[1]{\textbf{#1}}

\newcommand{\id}{\operatorname{Id}}

\newcommand{\ad}{\operatorname{ad}}

\newcommand{\jjap}[1]{\lfloor #1 \rfloor }
\usepackage{amsthm}

\RequirePackage{url}
\newtheorem{prop}{Proposition}[section]
    
      \newtheorem{them}{Theorem}[]
    \newtheorem*{thm*}{Theorem}
    \newtheorem*{cor*}{Corollary}
    \newtheorem*{teo normal}{"Twisted Conjugacy" Theorem}

    \newtheorem{lemma}{Lemma}
    \theoremstyle{remark}
     \newtheorem{ex}{Example}
\newtheorem{rmk}{Remark}[section]
\theoremstyle{definition}
\newtheorem{defn}{Definition}

\numberwithin{equation}{section}
\numberwithin{thm}{section}
\numberwithin{defn}{section}
\numberwithin{prop}{section}
\numberwithin{cor}{section}
\numberwithin{lemma}{section}
\numberwithin{rmk}{section}

\newcommand{\g}{\gamma}

\newcommand{\s}{{\sigma}}

\newcommand{\N}{{\mathbb N}}

\newcommand{\cC}{{\mathcal C}}

\newcommand{\cH}{{\mathcal H}}

\newcommand{\cK}{{\mathcal K}}
\newcommand{\cL}{{\mathcal L}}
\newcommand{\cM}{{\mathcal M}}
\newcommand{\cN}{{\mathcal N}}
\newcommand{\cO}{{\mathcal O}}

\newcommand{\cQ}{{\mathcal Q}}
\newcommand{\cR}{{\mathcal R}}
\newcommand{\cS}{{\mathcal S}}
\newcommand{\cT}{{\mathcal T}}

\newcommand{\fa}{{\bal}}
\newcommand{\fb}{{\bbt}}

\newcommand{\fm}{{\mathfrak{m}}}

\newcommand{\tk}{{\mathtt{d}}}
\newcommand{\tm}{{\mathtt{m}}}

\newcommand{\sob}{{\mathtt{w}}}

\newcommand{\tD}{{\mathtt{D}}}

\newcommand{\tK}{{\mathtt{K}}}

\newcommand{\tM}{{\mathtt{M}}}

\newcommand{\be}{{\bf e}}

\usepackage{bm}

\newcommand{\al}{{\alpha}}
\newcommand{\bt}{{\beta}}

\newcommand\norma[1]{\left\lVert#1\right\rVert}

\newcommand{\im}{{\rm i}}
\newcommand{\jap}[1]{\langle #1 \rangle}

\newcommand{\e}{{\varepsilon}}

\newcommand{\meas}{{\rm meas}}

\newcommand{\tw}{{\mathtt{w}}}
\newcommand{\nnorm}[1]{{\left\vert\kern-0.25ex\left\vert\kern-0.25ex\left\vert #1 
    \right\vert\kern-0.25ex\right\vert\kern-0.25ex\right\vert}}

\newcommand{\buu}{{u^\bal}{\bar{u}^\bbt}}

\newcommand{\modi}[1]{\abs{u_{#1}}^2}

\newcommand{\bal}{{\bm \al}}
\newcommand{\bbt}{{\bm \bt}}

\newcommand{\Hrp}{\cH_{r,p}}
\newcommand{\nore}[1]{\abs{#1}_{r,p}}

\newcommand{\zero}[1]{#1^{(0)}}
\newcommand{\zeroR}[1]{#1^{(0,\cR)}}
\newcommand{\zeroK}[1]{#1^{(0,\cK)}}
\newcommand{\due}[1]{#1^{(-2)}}
\newcommand{\buon}[1]{#1^{\ge 2}}
\newcommand{\call}{\mathcal{L}}

\usepackage{framed,enumitem} 

\newcommand{\dueK}[1]{#1^{(-2,\cK)}}

\newcommand{\crac}{{\mathtt C'}}

\renewcommand{\o}{{\omega}}

\newcommand{\betta}{{\omega}}

\newcommand{\ub}{v}
\newcommand{\zb}{z}

\newcommand{\bo}{{\nu}}
\newcommand{\bO}{{ \Omega}}

\begin{document}
\author{Luca Biasco}
\address{Università degli Studi Roma Tre}
\email{biasco@mat.uniroma3.it}

\author{Jessica Elisa Massetti}
\address{Università degli Studi Roma Tre}
\email{jmassetti@mat.uniroma3.it}

\author{Michela Procesi}
\address{Università degli Studi Roma Tre}
\email{procesi@mat.uniroma3.it}
 
 \title{On the  construction of Sobolev Almost periodic invariant tori for the 1d NLS}
\begin{abstract}
We discuss a method for the construction of  almost periodic solutions of the 
one dimensional  analytic NLS
with only Sobolev regularity 
both in time and space.
This is the first result  of this kind for PDEs. 
\end{abstract} 
\maketitle
\setcounter{tocdepth}{1} 

\section{Overview and main result}
In KAM theory for PDEs, one of the most challenging problems  of the last twenty years is the construction of \co{almost-periodic} solutions\footnote{i.e. solutions which are limit (in the uniform topology in time) of quasi-periodic solutions}. Very few results are known in this direction, and all of them investigate models with \co{external parameters} in order to avoid resonances and 
focus on the small divisor problem. In particular, in all the existent literature, the construction of almost periodic solutions is achieved at the cost of an \co{extremely high} decay of their Fourier coefficients, which approach zero 
  super-exponentially, exponentially or sub-exponentially (Gevrey).  Indeed, because of the fact that the classical KAM procedure is not uniform in the dimension, one cannot  naively construct a quasi-periodic solution supported on a $n$-dimensional invariant torus and then take the limit $n\to\infty$: one would fall on the elliptic fixed point. Relying on an NLS with multiplicative potential (producing an infinite set of free parameters) and smoothing nonlinearity, P\"oschel  (partially) tackles this problem in \cite{Poschel:2002}, by iteratively constructing its solutions though successive small perturbations of finite-dimensional tori, parametrized though action-angle variables and eventually characterized by a very strong compactness property: in order to overcome the dependence  of the KAM estimates on their dimension, the radii of these tori have to shrink super-exponentially, this leading to very regular solutions. See also \cite{GX13} for a generalization of P\"oschel's approach to the analytic cathegory, by using T\"oplitz-Lipschitz function techniques. \\ Bourgain understood that when the  dimension grows to infinity, action-angle variables become the Achilles' heel of the KAM procedure (they are not even well defined, in general) and in his pioneering work \cite{Bourgain:2005} on the quintic NLS with Fourier multipliers (providing external parameters in $\ell^\infty$), he proposed a different approach by working directly in cartesian coordinates, without introducing any action-angle variables, and relying on a Diophantine condition which is taylored for the infinite dimension. For most choices of the parameters, this lead to the construction of almost periodic invariant tori which support Gevrey solutions. 

The recent work \cite{BMP:almost} extends the techniques of \cite{Bourgain:2005} and proposes a novel, flexible approach which allows to construct in a unified framework,  both maximal and elliptic invariant tori of any dimension which are the support of the desired Gevrey solutions, for an NLS with Fourier multipliers which does not necessarily preserve momentum. The persistence of the invariant tori is achieved though an abstract normal form theorem "à la Herman", whose estimates are  \co{uniform in the dimension} (see \cite[Theorems 3 and 7.1]{BMP:almost}).  See also \cite{Berti-Montalto:memoirs, Massetti:ETDS, Massetti:APDE} 
for a survey on this technique.

The possibility of constructing almost-periodic solutions for a \co{fixed} PDE, i.e. "eliminating" the external parameters through amplitude-frequency modulation, appears to be intimately related to the regularity issues. Roughly speaking, a fast decay of the "actions" (needed for the scheme to converge) leads to a  weak  modulation of frequencies which in turn results in bad bounds on the small divisors.
It then becomes fundamental to look for  almost-periodic solutions in lower regularity spaces. However this appears to be a very difficult problem, due to the presence of extremely small divisors.
\\
An analogous problem with rapidly vanishing small divisors
arises in Birkhoff Normal Form theory for PDEs.
Indeed in the analytic or Gevrey case one has sub-exponential stability times
(see \cite{Faou-Grebert:2013} and \cite{Cong}),
whereas in the Sobolev case the best known estimates have a power growth in the Sobolev exponent
(see \cite{Bambusi-Grebert:2006}, \cite{FI}, \cite{Berti-Delort},\cite{BMP:2019}).
\\
The counterpart of total and long time stability results is the construction of unstable trajectories, which undergo  growth of the Sobolev norms, see \cite{Bou1,CKSTT, GuaKa, GHHMP, GGPP}.

In the context of quasi-periodic solutions there is a wide literature regarding solutions of finite regularity. However most of the interest is in the case of a non-linearity which is only Sobolev. The strategy is to apply a Nash-Moser scheme (generalizing \cite{CW}, \cite{Bo4}) and prove tame estimates on the inverse of the linearized equation at an approximate solution, either by multiscale methods, \cite{BB1}, or by reducibility \cite{BBM1}.   Of course one can apply this techniques also  in the case of analytic non-linearities. Another possibility is to construct solutions which are analytic in time and only finite regularity in space (see, e.g. \cite{Poschel:1996}, \cite{CFP}). Note  however that solutions obtained with such methods are often actually smooth (by bootstrap arguments).  
Unfortunately it is not at all clear how  (or even whether it is possible)  to  extend  such ideas to almost-periodic solutions. 
\\
A main difference with the quasi-periodic case is that for almost-periodic solutions
the topology of the phase space is crucial.
Indeed, in the quasi-periodic case  (at least for semi-linear PDEs)
one typically 
 looks for an {\sl analytic embedded finite} dimensional torus 
in a fixed phase space of $x$-dependent functions,
by modulating the analyticity strip in the angles. Then
the  analyticity in time directly follows.
On the other hand, it is not obvious at all (and clearly it strongly
depends on the topology of the phase space), whether
the  {\sl embedded\footnote{In the infinite-dimensional case, even the definition of such an embedding requires accuracy in the choice of the topology. Indeed most of the existent literature bypass this issue and directly construct the solutions as limit of quasi-periodic ones.} infinite} dimensional torus  is analytic.
Anyway, the analyticity in time does not follow
since the map $t\mapsto\o t$ is  not even continuous\footnote{At 
least if $\sup_j |\o_j|=\infty$ as it is typical in PDEs.}.
Generically, almost-periodic solutions are not continuous trajectories in the phase space.
However, such solutions  can be regular as complex valued functions of time and space,
depending on the regularity of the phase space. See Remark \ref{sona}.
\\
Taking all the advantage from the flexible construction proposed in \cite{BMP:almost}, we present here the very first result of persistence of almost-periodic solutions with finite regularity both in time and space.  We stress that our solutions are not maximal tori but instead are mostly localized on a {\it sparce lattice}. \\ 
The core of our strategy is that in constructing solutions mostly localized in such lattices we may impose {\it very strong} Diophantine conditions, see Definition \ref{diomichela}, so that our small divisors can be controlled similarly to the Gevrey case of \cite{BMP:almost}. The key  points are the definition of diophantine vectors \ref{diomichela}, the measure estimates of Theorem 2 and the bounds on the homological equation  in Lemma \ref{mah} .


We present in this note our strategy in the simplest possible setting. To this purpose, we consider families of NLS equations on the circle with external parameters of the form:
\begin{equation}\label{NLS}
\im u_t + u_{xx} - V\ast u +
f(|u|^2)u=0\,.
\end{equation}
Here $\im=\sqrt{-1}$,
$u=\sum_{j\in\Z}  u_j e^{\im j x},$
$V\ast$ is a Fourier multiplier 
\begin{equation}
\label{vu}
V\ast u = \sum_{j\in\Z} V_j u_j e^{\im j x}\,,\quad \pa{V_j}_{j\in\Z}\in 
[-\nicefrac12,\nicefrac12]^\Z\subset \ell^\infty(\R)
\end{equation}
and $f(y)$ is real analytic in $y$ in a neighborhood of $y=0$. We shall assume that $f(0)=0$. 
By  analyticity, for some $\mathtt R>0$ we have
\begin{equation}\label{analitico}
f(y)= 
\sum_{d=1}^\infty f^{(d)} y^d\,,\quad
|f|_{\mathtt R}:=\sum_{d=1}^\infty|f^{(d)}|\mathtt R^d <\infty \,,
\end{equation}
\\
We look for solutions in Fourier series
$
u(x) = \sum_{j\in \Z} u_j e^{\im j x},
$
where $\pa{u_j}_{j\in\Z}$
belongs to the scale of Banach spaces 
\begin{equation}\label{pecoreccio}
\sob_{p}
:= 
\set{u:= \pa{u_j}_{j\in\Z}\in\ell^2(\C)\; : \quad \norm{u}_{p}:= \sup_{j\in\Z}\abs{u_j} \jjap{j}^{p}< \infty}\,, \quad p>1\,,
\end{equation}
where\footnote{Obviously one could also take the more standard weight $\jap{j}:= \max\{1, |j|\}$, which
generates the same Banach space. 
We made such choice for merely technical reasons.} 
$\jjap{j} := \max\{2, |j|\}$.
As it is common habit, we endow $\sob_{p}\subset \ell^2$ with the symplectic structure inherited from $\ell^2$. 
Note that $u(x)\in C^k$ for every $k<p-1.$

It is well known that \eqref{NLS}
is a hamiltonian 
system  with Hamiltonian 
\begin{equation}\label{Hamilto0}
H_V(u):=\sum_{j\in\Z} (j^2+V_j) |u_j|^2+P\,,\quad \mbox{with}\quad
 P:= \int_\T F(x,|\sum_j u_j e^{\im j x}|^2) dx \,,\quad F(y):=\int_0^y f(s) ds\,.
\end{equation}

\medskip


\noindent
We  consider as ``tangential sites'' the following (infinite) subset of $\Z$ 
\begin{equation}
\cS :=  \set{ 2^h,\quad h\in\N} \,,\quad \Z= \cS\cup \cS^c
\end{equation}
and define the set of ``tangential frequencies''
\begin{equation}
\label{panc}
\pan_\cS := \set{\nu = (\nu_j)_{j\in\cS}\in \R^\cS\, : \, \abs{\nu_j- j^2} < \frac12}.
\end{equation}
The main result is the following
\begin{them}\label{toro Sobolevp}
	Consider a translation invariant NLS Hamiltonian 
	as in \eqref{NLS}.
	For any $p>1$,  $\g>0$  there exists $\e_*=\e_*(p)>0$ such that,
	for all $r>0$
	satisfying 
	\begin{equation}
	\label{cornettonep}
	\e:=\frac{\abs{f}_{\mathtt R}}{\g \mathtt R} r^2 \le \e_*
	\end{equation}
	for every 
	\begin{equation}\label{pifferello}
\sqrt I\in { B}_{r}(\sob_{p})\quad
\mbox{with}\quad I_j=0\quad \mbox{for}\quad j\in\cS^c
\end{equation}
	and 
	for any    $W\in [-\nicefrac14,\nicefrac14]^{\cS^c}$  such that $W_0\ne 0$,
	the following holds.
	\\
	There exist
	a positive measure Cantor-like   set
	$\cC\subset \cQ_{\cS}$,  such that for all $\nu\in \cC$
	there exists a potential  $V\in  [-\nicefrac12,\nicefrac12]^\Z$ 
	and a analytic change of variables  $\Phi:{ B}_{2r}(\sob_{p})\to { B}_{4r}(\sob_{p})$  such that 
	\begin{equation}\label{pippero}
	\cT_I:=\{u\in \tw_p: |u_j|^2= I_j \quad \forall j\in Z\}
	\end{equation}
	 is an elliptic  KAM torus of frequency $\nu$ for
	$
	H_{V} \circ \Phi$ with $V_j
	=W_j$ for $j\in\cS^c$. 
	Finally $V$  depends on $\nu$
	in a Lipschitz way. 
\end{them}

\begin{rmk}\label{sona}
Since we are able to construct solutions for every 
	$\sqrt I\in {\bar B}_{r}(\sob_{p})$  with 
	$I_j=0$ for $j\in\cS^c$  and we have that $|u_j(t)|\sim \sqrt I_j$, then 
 most of our  solutions  have only Sobolev regularity (both in space and time)
 and are not analytic, Gevrey or even $C^\infty.$ To be more explicit, the map $$\imath: \T^\infty \to \cT_I\subset \tw_p,\quad \varphi=(\varphi_j)_{j\in\Z}\mapsto (\sqrt{I_j} e^{\im \varphi_j} )_{j\in\Z}$$
 is analytic provided that we endow $\T^\infty$ with the $\ell^\infty$-topology. On the other hand, the map $\psi : \R \to \T^\infty,\, t\mapsto \o t$ is not even continuous. Of course the regularity of $\imath \circ \psi$ depends on the choice of the actions $I_j$. If we take for instance $I_j = \jap{j}^{-p}$, then $\imath \circ\psi : \R\to \tw_p$ is not continuous. On the other hand, for all $x\in\T$, the map  $$t\, \mapsto \,\imath \circ \psi(t,x) = \sum_j \jap{j}^{-p} e^{\im (\omega_j t + j x)}$$ is $C^k(\R,\C)$ for all $k < \frac{p-1}{2}$.
\end{rmk}

\begin{rmk}
The choice of $\cS$ is rather arbitrary.
However the fact that $\cS$ is sparse is essential.
More precisely we need a growth of the type
$\cS:=(s_h)_{h\in\N}$, $s_h\sim e^{\ln^3 h}.$ 
This and other extensions will be discussed in the forthcoming paper
\cite{BMP:almostsob}.
\end{rmk}

We can be rather explicit in our description of the set $\cC$. We start by 
fixing the hypercube 
\begin{equation}
\label{pan}
\pan := \set{\omega = (\omega_l)_{l}\in \R^\Z\, : \, \abs{\omega_l- l^2} < \frac12}\, , \qquad \cQ=\cQ_{\cS}\times \cQ_{\cS^c}
\end{equation}
and by introducing the following 
\begin{defn}[Diophantine condition] \label{diomichela} Let $\tau \ge 2$. We say that a vector $\omega\in\pan$ belongs to ${\mathtt D}_{\g,\cS}$ if it satisfies
	\begin{equation}\label{diofantino}
	\tag{$\mathtt{DC}$}
	|\omega\cdot \ell| \ge  \gamma \prod_{j\in \cS}\frac{1}{(1+|\ell_j|^2\jap{\log_2 j}^{2})^\tau}\,,\quad \forall  \ell : 0<|\ell|<\infty \,,\quad \sum_{j\in \cS^c}|\ell_j|\le 2 \,,\quad \pi(\ell)= \fm(\ell)=0
	\end{equation}
	where $\pi(\ell) :=\sum_{j\in \Z} j\ell_j$ and $\fm(\ell):=\sum_{j\in \Z} \ell_j$.
\end{defn}

\begin{them}\label{torobolev}
	Under the hypotheses of Theorem \ref{toro Sobolevp},
	there exist $C>1$ and a Lipschitz map 
	\begin{equation}
	\label{omegone}
	\bO: \cQ_\cS \to \cQ_{\cS^c}\,,\qquad |{\bO_j-j^2-W_j}| +\g  |\bO_j|^{\rm Lip} \le C \g \epsilon\quad\forall j\in \cS^c\,,
	\end{equation}

	such that 
	\[
	\cC:= \{ \nu\in \cQ_{\cS}: \quad (\nu,\Omega(\nu))\in {\mathtt D}_{\g,\cS} \}\,
	\]
	where $\abs{\cdot}^{\rm Lip}$ stands for the classical Lipschitz semi-norm.  Moreover $\rm{meas} (\cQ_S \setminus \cC) \le C\g$.
\end{them}

\begin{rmk}\label{tana}
	Note that \eqref{diofantino} is a much stronger diophantine condition that the one proposed in \cite{Bourgain:2005} (or \cite{BMP:almost}), where the denominators were of the form $1+|\ell_j|^2 j^2$. Of course the reasons why we can impose such strong diophantine conditions, still 
	obtaining a positive measure set, are the structure of the set $\cS$ and the fact that we only need to consider denominators with $\sum_{j\in \cS^c} |\ell_j|\le 2$. 
	Note that  our diophantine condition becomes the usual one 
	in the case $\sum_{j\in \cS^c} |\ell_j|=0$,
	by just renaming the indices $j= 2^h$. To deal with the remaining terms ($\ell$ not  supported only on the tangential sites) we use the constants of motion and the dispersive nature of the equation ($\omega_k\sim k^2$).
\end{rmk}

 \subsection{Functional Setting}
We start by endowing our phase space $\tw_p$ (see \eqref{pecoreccio})
with the symplectic structure 
$\Omega=\im  \sum_j d u_j\wedge d\bar u_j$ induced by the Hermitian product on $\ell^2=\ell^2(\Z,\C)$.
In order to introduce the Hamiltonians on such space we introduce the following
\begin{defn}[Multi-index notation]\label{baccala}
  In the following we denote, with abuse of notation, by $\N^\Z$ the set of 
   multi-indexes $\bal,\bbt$ etc. such that
  $|\bal|:=\sum_{j\in\Z}\bal_j$ is finite.
  As usual $\bal!:=\prod_{j\in\Z,\, \bal_j\neq 0}\bal_j$.
  Moreover $\bal\preceq\bbt$ means $\bal_j\leq\bbt_j$
  for every $j\in\Z$, then $\binom{\bbt}{\bal}:=\frac{\bbt!}{\bal!(\bbt-\bal)!}.$
  We also define $u^\bal:= \prod_{j\in\Z}u_j^{\bal_j}$ (which is a finite product due to the condition $|\bal|<\infty$).
   Finally take $j_1<j_2<\ldots<j_n$
  such that $\bal_j\neq 0$ if and only if $j=j_i $ for some $1\leq i\leq n$,
  as usual we set $\partial^{\bal} f:=\partial^{\bal_{j_1}}_{u_{j_1}}
  \ldots \partial^{\bal_{j_n}}_{u_{j_n}} f
 \,;$
  analogously for $\partial_{\bar u}^\bbt f.$
  \end{defn}
  Following \cite{BMP:almost} (see Definition 2.1 with $a=s=0$) we introduce the space of regular Hamiltonians $ \cH_r(\tw_p)$.
\smallskip
\begin{defn}[Regular Hamiltonians]\label{Hr} 
	Consider a formal
	 power series expansion
	\begin{equation}\label{mergellina}
H(u)  = \sum_{(\bal,\bbt)\in \cM }H_{\bal,\bbt}u^\bal \bar u^\bbt\,,
	\qquad
	u^\bal:=\prod_{j\in\Z}u_j^{\bal_j}\,,
\end{equation}
where
\begin{equation}\label{zorro}
\cM:=\left\{
(\bal,\bbt)\in \N^\Z\times\N^\Z\,, \ \ {\rm s..t.}\ \ 
|\bal|=|\bbt|<+\infty\,,\ \ 
\sum_{j\in \Z} j (\bal_j-\bbt_j)=0
\right\}
\end{equation}
satisfying the reality condition
		\begin{equation}\label{real}
		H_{\bal,\bbt}= \overline{ H}_{\bbt,\bal}\,, \qquad
		\forall\, (\bal,\bbt)\in \cM\,.
		\end{equation}
	We denote by $\mathcal{H}_{r,p}$ for $p>1,$ $r>0$ the space of \emph{regular} Hamiltonians, i.e. those $H$ such that
		\begin{equation}\label{norma1}
 \abs{H}_{r,p}
 :=
\frac12  \sup_j  
\sum_{(\bal,\bbt)\in \cM} \abs{H_{\bal,\bbt}}\pa{ \bal_j + \bbt_j}u_p^{\bal + \bbt - 2e_j}
<\infty\,,
 \end{equation}
  where $u_p=u_p(r)$ is defined as
 \begin{equation}\label{giancarlo}
 u_{p,j}(r):= r  \jjap{j}^{-p}  \,.
 \end{equation}
	\end{defn}

\begin{rmk}
Regarding $\cM$ in \eqref{zorro} the condition 
$|\bal|=|\bbt|$, i.e. $\fm(\bal-\bbt)=0$, corresponds to mass conservation,
namely 
 the $H$ Poisson commutes with the { \sl mass} $\sum_{j\in \Z}|u_j|^2$;
 moreover
 $\sum_{j\in \Z} j (\bal_j-\bbt_j)=0$, i.e.   $\pi(\bal-\bbt)=0$, corresponds to
	 momentum conservation, 
	 namely $H$ Poisson commutes with the 
		{ \sl momentum} $\sum_{j\in \Z}j|u_j|^2$.
\end{rmk}


\noindent
	Note that 
$\abs{\cdot}_{r,p}$ is a seminorm
on  $\mathcal H_{r,p}$ and a norm on its subspace
\begin{equation}\label{lenticchia}
 \mathcal H_{r,p}^0:=
 \{\
		H\in\mathcal H_{r,p}\ \ {\rm with}\ \ 
H(0)=0\ \}\,,
\end{equation}
endowing $\mathcal H_{r,p}^0$ with a Banach space structure.\\

\medskip

\noindent
To control the Lipschitz dependence on the frequency throughout the iterative scheme, we define the following weighted Lipschitz semi-norm. Fix $\g>0$  
and assume that $H = H(\omega)\in\mathcal H_{r,p}$ for every $\omega\in\dg$.  We then define the semi-norm
(as usual $
|v|_\infty:=\sup_{j\in\Z} |v_j|
$)
\begin{eqnarray}
\norma{H}_{r,p}
&:=& 
\sup_{\omega\in\dg} \nore{H(\omega)} +\g
\sup_{\substack{\omega,\omega'\in\dg\\ \omega\neq \omega'}} 
\frac{\nore{H(\omega) - H(\omega')}}{\abs{\omega - \omega'}_{\infty}}
<\infty\,.
\label{normag}
\end{eqnarray} 
Finally we set
$$
\mathcal H_{r,p}^{\rm Lip}:=\Big\{
H(\omega)\in\mathcal H_{r,p},\ \ \omega\in\dg\,,\ \ \ {\rm with}\ \ \norma{H}_{r,p}<\infty
\Big\}\,,
\qquad
\mathcal H_{r,p}^{\rm{Lip},0}:=\left\{
H\in\mathcal H_{r,p}^{\rm Lip} 
\ \ {\rm with}\ \ 
H(0)=0
\right\}\,.
$$
It is immediate to verify that 
$\mathcal H_{r,p}^{\cO,0}$ is a Banach space
endowed with the above norm.
\begin{prop}[Monotonicity]
The norm $\|\cdot\|_{r,p}$  is monotone decreasing  in $p$ and monotone increasing in $r$: 
\begin{equation}\label{che palle}
 \|\cdot\|_{r,p+\delta}\leq \|\cdot\|_{r + \rho,p} \quad \forall\, \rho, \delta>0.
 \end{equation}
 \end{prop}
The fact that this norm is increasing in $r$ follows directly from mass conservation and the fact that $H(0)=0$. Concerning the monotonicity in $p$, we refer the reader to \cite[Proposition 6.3]{BMP:2019}, where the proof is contained, written in the case of $| \cdot |_{r,p}$. The fact that it holds also in the Lipschitz frame, follows trivially. 

\smallskip

Finally, as it is expected, this norm also behaves well with respect to the Hamiltonian flows and Poisson brackets. See\footnote{Our norm corresponds to the norm $\abs{\cdot}_{r, 0 , \tw}$ in the notation of \cite{BMP:2019}, with $\tw = (\tw_j(p)) \equiv (\jjap{j}^p)$. } \cite{BMP:2019} Proposition 2.1 and Lemma 2.1.

\section{Normal forms and proof of Theorem \ref{toro Sobolevp} and \ref{torobolev}}\label{proiezioni}

The seminal idea contained in \cite{Bourgain:2005} for proving the persistence of a full dimensional invariant torus, consisted in smartly rewriting $H_{\rm V}$ in a way that one could select those terms preventing the torus to be invariant for its dynamics. This idea has been formalized in \cite{BMP:almost}, in terms of a 
 {\it degree decomposition} with increasing order of zero at $\cT_I$, defined for any regular Hamiltonian. For convenience of the reader, we sketch here the main features needed to prove our persistence theorem, while for detailed statements and proofs we address the reader to \cite[Section 4]{BMP:almost}.
\\
We want to prove that, in suitable variables,  $\cT_I$ introduced in \eqref{pippero} is an
 invariant torus on which the flow is linear with frequencies
 $\omega$. To this purpose  we introduce 
 a suitable degree decomposition, whose
main idea  is to make a \co{power series expansion centered    at $I$} without introducing a singularity in order to highlight the terms which prevent $\cT_I$ to be invariant of frequency $\omega$.
Consider a Hamiltonian $H(u)$ expanded
in Taylor series at $u=0$ and  tautologically rewrite $H$ as 
	\begin{equation}\label{vitello}
H= \sum_{\substack{m,\al,\bt\in \N^\cS\\ \al\cap \bt= \emptyset\\ a,b\in \N^{\cS^c}}} H_{m,\al,\bt,a,b}|\ub|^{2m}\ub^\al\bar \ub^\bt \zb^a\bar \zb^b
\end{equation}
	where, by slight abuse of notation\footnote{Consisting in a reordering of the indexes $j$.},  $u=(v,z)$ with $\ub= \pa{v_j}_{j\in \cS}:=\pa{u_j}_{j\in \cS}$
and 
$\zb=\pa{\zb_j}_{j\in \cS^c}:=\pa{u_j}_{j\in \cS^c}$.
\\
Then introduce the auxiliary ``action'' variables
 $w=(w_j)_{ j\in \cS}$ 
 substituting 
$|v|^{2m} v^\al \bar v^\bt z^a \bar z^b \rightsquigarrow w^m v^\al \bar v^\bt z^a\bar z^b$ in \eqref{vitello}.
Now we  Taylor expand the Hamiltonian
 with respect to $w$ and $z$ at the point
$w_j=I_j$ for $j\in \cS$ and $z=0$ respectively.

\begin{defn}[Degree decomposition] Let $I$ be fixed as in \eqref{pifferello}. For every integer $d\ge -2$ and any regular Hamiltonian $H\in\cH_{r,p}$ we define the following projection:

	\begin{equation}\label{grado vero}
	\Pi^d H := H^{(d)} := \sum_{\substack{m,\al,\bt,\delta\in \N^\cS, a,b\in\N^{\cS^c}\\\al\cap \bt= \emptyset\,,\; \delta\preceq m \\ 2|\delta|+|a|+|b|= d+2} } {H}_{m,\al,\bt,a,b}
	\binom{m}{\delta} I^{m-\delta} (|\ub|^2- I)^\delta \ub^\al \bar \ub^\bt \zb^a \bar \zb^b\,.
	\end{equation}
	where $\delta \preceq m$ means that $\delta_j \leq m_j$ for any $j\in\cS$, $|v|^2 =\pa{|v_j|^2}_{j\in\cS}$, while the multiindex notations are introduced in Definition \ref{baccala}.
\end{defn}
In this way, if $\cS=\Z$, projections coincide with the ones of Section 4 of \cite{BMP:almost}, while if $\cS= \emptyset$, $H^{(d)}$ represents the usual homogeneous degree at $\zb=0$.
\\
In this way, given $H\in\cH_{r,p}$, then 
\begin{equation}\label{semicroma}
H = H^{(\le 0)} + H^{(\ge 1)} \equiv H^{(-2)} + H^{(-1)} + H^{(0)} + H^{(\ge 1)}  
\end{equation}
 where $H^{(-2)}$ consists of terms which are constant w.r.t. both $z$ and and the "auxiliary action" $w = \abs{v}^2$, $H^{(-1)}$ is independent of the action but linear in the $z_j$, while  $H^{(0)}$ contributes with two terms: the one linear in the action and independent of $z$, the second one quadratic in $z$ and independent of the action. Finally, $H^{(\ge 1)} $ is what is left and $X_{H^{(\ge 1)}}$ vanishes on $\cT_I$. 
 \smallskip
 
 The operators $\Pi^d$ define continuous projections satisfying $\Pi^d \Pi^d = \Pi^d $ and
	$\Pi^{d'} \Pi^d=\Pi^d \Pi^{d'} =0$
	for every $d'\neq d$,  $d'\geq -2.$ Moreover, this decomposition enjoys all the crucial properties required for a KAM scheme to converge, in particular they behave well with respect to Poisson brackets, that is: 
	$$
	\forall F, G \in\cH_{r,p}\quad  \set{F, G^{\ge 1}}^{(-2)} = 0 
	$$
and
\begin{equation*}
\label{woodstok}
  \quad F^{(-2)}=0 \, \Longrightarrow \, \set{F, G^{\ge 1}}^{(-2)} = 0, \quad \text{and} \quad F^{( -1)} = 0 = F^{ (-2)} \, \Longrightarrow  \set{F, G^{\ge 1}}^{(\le 0)} = 0.
\end{equation*}
For all the properties of the projections see \cite{BMP:almost} Proposition 4.1 and 4.2.

\smallskip

\noindent

%

\noindent
\begin{defn}[Normal Forms] Let $I,r,p$ be  as in \eqref{pifferello}. Let $D: \cQ \to \R$ be the linear map defined as

 \[
 D(\o):= \sum_{j\in\Z} \omega_j |u_j|^2\,.
 \]
 We will say that   a Hamiltonian $N$ is  in normal form at  $\cT_I$ with frequency $\o$
if $N-D(\o) \in \cH_{r,p}^{{(\ge 1)}}$.
We denote the affine space of such Hamiltonians by $\cN_{r, p}(\omega;I).$ 
\end{defn}

\smallskip

\noindent
\begin{them}
	\label{allaMoserbis} 
Consider 
$r_0,p_0, \rho, r,\delta>0$ with \begin{equation}
\label{newton}
\rho\le r_0\nicefrac{}{2}\quad \mbox{ and}\quad r \le r_0\nicefrac{}{2\sqrt{2}}\,.
\end{equation}
There exists $\bar{\epsilon},\bar C>0$, depending only on
$
\rho/r_0$
and
$\delta$ 
such that the following holds. 
Let $\sqrt{I}\in {\bar B}_{r}(\tw_{p_0+\delta})$ be such that $I_j = 0,\ \forall j \notin \cS$.
Consider a family of normal forms 
  $N_0(\o;I)\in\cN_{r_0, p_0 }(\o;I)$.
  Finally consider a Lipschitz family of Hamiltonians  
$H(\o),$   
   with $\omega\in D_{\g, S}$,  assume that   $H(\o)- D(\o) \in\cH_{r_0,p_0}$ and
\begin{equation}\label{maipendanti}
(1+\Theta)^4 \epsilon \le \bar{\epsilon} \,,\qquad
\mbox{where}\qquad
\epsilon:=\g^{-1}\norma{H- N_0}_{r_0 ,p_0 }\,,
\quad \Theta = 
\g^{-1}\norma{D- N_0}_{r_0 ,p_0}\,.
\end{equation}
Then there exist a symplectic diffeomorphism $\Psi: {B}_{r_0-\rho}\pa{\tw_{p_0 + \delta}} \to {B}_{r_0}\pa{\tw_{p_0 + \delta}} $, close to the identity, a unique correction (counter term) 
$\Lambda = \sum_{j}\lambda_j \pa{\modi{j} - I_j}$,
Lipschitz depending on $\o\in\dg$, namely
\begin{equation}\label{loredana}
 |\lambda(\omega)|_\infty +\g 
 \frac{\abs{\lambda(\omega) - \lambda(\omega')}_\infty}{\abs{\omega - \omega'}_{\infty}}
 \le \bar C \g   (1+\Theta)^2\epsilon\,,\qquad
 \forall\, \omega,\omega'\in\dg,\ \  \omega\neq \omega'
\end{equation}
and a family of normal forms $N(\o;I)\in\cN_{r_0-\rho,p_0 + \delta}(\o;I)$,  such that 
\begin{equation}\label{coniugio}
\pa{\Lambda + H}\circ \Psi = N.  
\end{equation}
\end{them}

\subsection{Proof of Theorem \ref{toro Sobolevp} and \ref{torobolev}}\label{auricchio} 
Theorem \ref{toro Sobolevp} follows from Theorem \ref{allaMoserbis} in a straightforward way.
One first rewrites $H_V$ in \eqref{Hamilto0} as $D + P +\Lambda$ by setting  $\lambda_j = j^2-\omega_j + V_j$ so that it fits the hypotheses with $N_0=D$, $\Theta=0$ and $\epsilon\sim \varepsilon$
(recalling \eqref{cornettonep} and taking $\e_*$ small enough).  Then Theorem \ref{allaMoserbis} gives us the desired change of variables provided that $\Lambda$ is fixed in terms of the frequency $\omega$. Now we denote $\omega_j =\nu_j$ if $j\in \cS$ and $\omega_j=\Omega_j$ otherwise.
We get the equations
\begin{equation}
\label{equaOme}
\begin{cases}
&\bO_j+ \lambda_j(\bo, \bO) = j^2+W_j \,,\quad \mbox{if }\; j\notin \cS\\
&\bo_j +\lambda_j(\bo, \bO)= j^2+V_j \,,\quad \mbox{if }\; j\in \cS\,.
\end{cases}
\end{equation}

Now we   Lipschitz extend the map $\lambda: \,\dg \to \ell_\infty$   to the whole square $\pan$ and, by
\eqref{loredana} and
 the 
 Contraction Lemma, we solve the first equation finding $\bO= \bO(\bo)$. 
Finally we solve the second equation by setting 
$
V_j=\bo_j +\lambda_j(\bo, \bO(\bo))- j^2
$
for $j\in \cS$.
This concludes the proof of Theorem \ref{toro Sobolevp} and also shows \eqref{omegone}. Finally, from Lemma \ref{misurino} below Theorem \ref{torobolev} also follows.

\begin{lemma}[Measure estimates]\label{misurino}
	 The measure of $\cQ_\cS\setminus\cC$ is of order $\g$.
\end{lemma}
See Appendix \ref{measure} for the proof.
\section{small divisors and homological equation}

The proof of Theorem
\ref{allaMoserbis} 
is based on an iterative scheme that 
kills out the obstructing terms, namely  terms belonging to 
$ \cH_{r,p}^{(-2)},\cH_{r,p}^{(-1)}$ and $ \cH_{r,p}^{(0)}$,
by solving homological equations of the form 
\begin{equation}\label{esperidi}
L_\omega F^{(d)} = G^{(d)},\qquad G^{(d)}\in \cH_{r,p}^{(d)},\quad d= -2,-1,0.
\end{equation}
where 
\[
L_\omega[\cdot]:= \{\sum_j \omega_j|u_j|^2,\cdot \}\,,\qquad L_\omega F= \sum_{\bal,\bbt\in\N^\Z }\omega \cdot \pa{\bal-\bbt}F_{\bal,\bbt}u^\bal \bar u^\bbt\,.
\]
The convergence of the iterative KAM scheme comes from a good control of the solution $L_\omega^{-1} G^{(d)}$, which is discussed in detail in this section. On the other hand the iterative algorithm required to prove Theorem \ref{allaMoserbis} is the one of \cite{BMP:almost}. In Appendix \ref{AAA} we state the iterative Lemma \ref{iterativo}, we give a sketch of its proof for completeness and conclude by deducing Theorem \ref{allaMoserbis}.

Let us go back to the Homological equation $L_\o F = G$. As it is standard we denote by $\cH_{r,p}^\cK$ the kernel of $L_\omega$ and set
\begin{equation}
\label{ragno}
\Pi^\cK H := \sum_{\bal\in\N^\Z}H_{\bal,\bal}|u|^{2\bal}\,,\qquad \Pi^\cR H := H- \Pi^\cK H\,.
\end{equation}
Correspondingly, we define the following subspaces of $\Hrp$:
\begin{equation}\label{dolcenera}
\Hrp^\cK:=\{ H\in \Hrp\,:\quad \Pi^\cK H= H\}
\,,\qquad
\Hrp^\cR:=\{ H\in \Hrp\,:\quad \Pi^\cR H= H\}
 \,.
\end{equation}
these projections are continuous on $\Hrp$.
\\
On the subspace $\Hrp^\cR$, the Lie derivative operator $L_\omega$ is formally invertible with inverse 
\begin{equation}
\label{lala}
L_\omega^{-1} G := \sum \frac{G_{\bal,\bbt}}{\im\pa{\omega \cdot(\bal-\bbt})}\buu
\end{equation}
A good bound for the solutions of the homological equations \eqref{esperidi}
is a consequence of the following crucial 
\begin{lemma}[Straightening the tangential dynamics]\label{mah}
Let $\omega\in \tD_{\g,\cS}$ and $\omega\mapsto f(\omega)\in \cH_{r,p}^{\le 0,\cR}$ be a Lipschitz family of maps, then
there exists a constant $c>0$ such that
for all  $0<\delta<1$
\begin{equation}\label{laser}
	\|L_\omega^{-1} G\|_{r,p+\delta} 
	\le c\g^{-1} e^{\frac{c}{\delta} 
	\ln^2(\frac1\delta)}\|G\|_{r,p},
	\end{equation}
	for some suitable pure constant $c>0$.
\end{lemma}
\begin{rmk}
The proof of this Lemma is the real core of  our result. It is simple if $G$ is supported
only on $\cS$ or $\cS^c$. The crucial point is to control the interaction 
between tangential and normal sites.
The key ingredient is that we are only considering Hamiltonians that are { at most quadratic} in the normal variables, which in turn are supported on a sparse lattice (recall Remark
	\ref{tana}). 
		This result should be compared with the corresponding one in \cite{BMP:2019} Proposition 7.1
		item ($\mathtt M$). In this latter paper in order to control $L_\omega^{-1} G$ we can not take any $\delta>0$ but have to require $\delta \ge \tau_1$ instead, where $\tau_1$ is some fixed quantity. 
\end{rmk}

\begin{proof}[Proof of Lemma \ref{mah}] By \cite{BMP:2019},  Proposition $4.2$ and formulas $(4.10)-(4.11)$,
one has \begin{equation}\label{gulasch}
\|L_\omega^{-1} G\|_{r,p+\delta} 
\le \g^{-1} K_0  \|G\|_{r,p},
\end{equation}
where (see also \cite{BMP:2019} Proposition 7.1
item ($\mathtt M$))
\begin{equation}\label{tacchinobis}
	K_0 = \sup_{\substack{q\in\Z, \fa\neq \fb  \in\N^\Z \\ \fa_q + \fb_q \neq 0\, , \pi(\fa,\fb) = 0\,, \fm(\fa - \fb)=0 \\
			\sum_{j\in \cS^c}\fa_j+\fb_j\le 2\\
			\abs{\sum_i{\pa{\fa_i-\fb_i}i^2}}\le 2 \sum_j\abs{\fa_j-\fb_j}}} \pa{{\frac{\jjap{q}^2}{\Pi_j \jjap{j}^{\fa_j + \fb_j}}} }^\delta \frac{\g}{\abs{\omega\cdot\pa{\fa-\fb}}}.
	\end{equation}


	Let us set
	$$
	\al_h=\bal_{2^h}\,,\qquad
	\bt_h=\bbt_{2^h}\,,\quad h\in \N.
	$$
	By the conservation of mass and momentum and the constraints $\bal\ne \bbt$  and $\sum_{j\in \cS^c}\fa_j+\fb_j\le 2$, there exists at least one $j\in \cS$ such that $\fa_j+\fb_j\ne 0$. We denote the largest $j\in \cS$ with this property as $2^{h_\tM}$.  If $h_\tM=0$, by conservation of mass
	$|\fa_{1}-\fb_{1}|\le 2$, then  the Diophantine condition implies that $|\omega\cdot (\fa-\fb)|\ge \g/4$.
	Hence we assume that $h_\tM>0$.  
	\\
	By the conservation of mass and momentum  and by Lemma \ref{luchino}  with $a=1/2$, we can write 
	\begin{equation}
	\frac{\jjap{q}^2}{\Pi_j \jjap{j}^{\fa_j + \fb_j}} \le \frac{\na_1}{\prod_{l\ge 2} \na_l} \le \frac{\sum_{l\ge 2}\na_l}{\prod_{l\ge 2} \na_l} \le \frac{1 + \na_2^{\frac12}}{\Pi_{l\ge 2} \na_l^{\frac12}} \le \frac{2}{\Pi_{l\ge 3}  \na_l^\frac12}\,,
	\end{equation}
which we substitute in \eqref{tacchinobis} and get
	\begin{equation}\label{semper}
	K_0 \le \sup_{ \fa, \fb } \pa{\frac{2}{\Pi_{l\ge 3}  \na_l^{\frac12}}}^\delta \prod_{i\in\cS} \pa{1 + \jap{\log_2{i}}^2 \abs{\fa_i - \fb_i}^2}^\tau \g \,.
	\end{equation}
	We divide the proof in several cases depending on appropriate constraints on $\bal,\bbt$. We shall denote by $K_0^{(m)}$ the supremum  \eqref{semper} restricted to $\bal,\bbt$ in case $m$. 
	\subsection*{Case  1.   $\na_2> 2^{h_{\tM}}$} Here, since there are at most two normal sites we get\footnote{where one has $\jap{h}$ at the exponent since $\jjap{2^h} = 2^{\jap{h}}\, \forall h\in\N$ .}
	\[
	\Pi_{l\ge 3}  \na_l = \prod_{h\le h_\tm} 2^{\jap{h}(\al_h + \bt_h)}\,.
	\]
	Recalling that
	\[
	\jap{\log_2i}\abs{\fa_i - \fb_i}= { \jap{h} \abs{\al_h - \bt_h}}
	\le  \jap{h}\pa{\al_h + \bt_h}  \,,\quad \mbox{if}\quad i= 2^h
	\]
	we get
	\begin{equation}\label{vvv}
	K_0^{(1)}\le \exp{\pa{-\frac{\delta \log 2}{2} \sum_h \jap{h}(\al_h + \bt_h) + \tau \sum_h \log\pa{1 + \jap{h}^2(\al_h + \bt_h)^2}}}\,.
	\end{equation}
	
	Let $k_h := \al_h + \bt_h$ and $\bar\delta:= \frac{\delta\log 2}{2\tau}$, then 
	
	\begin{equation}
	K_0^{(1)}\le \exp{\set{ \tau\pa{\sum_h - \bar\delta \jap{h} |k_h| + \sum_h \log\pa{1 + \jap{h}^2k_h^2}}}}. 
	\end{equation}
	By Lemma \ref{pesto}, whenever $\jap{h} |k_h| \ge \frac{4}{\bar\delta} \log\frac{1}{\bar\delta}$, the exponent is negative. To get the desired bound, it remains to control
	\begin{align}\label{bibi}
	K_0^{(1)} &\le \exp{\set{ \tau\pa{\sum_{h: \jap{h} |k_h| 
					\le \frac{4}{\bar\delta} \log\frac{1}{\bar\delta}}  - \bar\delta \jap{h}  |k_h| + \log\pa{1 + \jap{h} ^2k_h^2}}}} \nonumber\\
	&\le \exp\set{\tau\sum_{h: \jap{h} |k_h| \le \frac{4}{\bar\delta}\log\frac{1}{\bar\delta}} \log(1 + \frac{16}{\bar\delta^2} \log^2(\frac{1}{\bar\delta}))}
	\le \exp\set{\tau \sum_{h: \jap{h} |k_h| \le \frac{4}{\bar\delta}\log\frac{1}{\bar\delta}} 6\log\frac{1}{\bar \delta}} \\
	&\le \exp\pa{\tau \frac{24}{\bar\delta} \log^2{\frac{1}{\bar\delta}}} = \exp\pa{\frac{C}{\delta}\log^2\frac{1}{\delta}}. \nonumber
	\end{align}
	where $C = \frac{48 \tau^2}{\log 2}. $\\
	
	\subsection*{Case 2.  $\na_1 > 2^{h_\tM} = \na_2$ and
		only one normal site}
	
	We have $\sum_{j\in \cS^c} \fa_j + \fb_j =1$ and the normal site must be $\na_1$. Moreover
	\[
	\Pi_{l\ge 3}  \na_l =  2^{h_\tM(\al_{h_\tM} + \bt_{h_\tM}-1)} \prod_{h < h_\tM} 2^{\jap{h}(\al_h + \bt_h)}\,,
	\]
	so 
	\begin{equation}
	\label{qulo}
	K^{(2)}_0\le \sup_{ \fa, \fb } \pa{\prod_{h< h_{\tM}} 2^{\jap{h}(\al_h + \bt_h)}}^{-\frac{\delta}{2}} 
	\pa{2^{{h_{\tM}}(\al_{h_{\tM}} + \bt_{h_{\tM}} -1)}}^{-\frac{\delta}{2}} \pa{1 + {h_{\tM}}^2 \abs{\al_{h_{\tM}} - \bt_{h_{\tM}}}^2}^\tau\prod_{h<h_{\tM}} \pa{1 + \jap{h}^2 \pa{\al_h + \bt_h}^2}^\tau \g^{-1} .
	\end{equation}
	\gr{(a)} If $\al_{h_{\tM}}=\bt_{h_{\tM}}$ and hence $h_\tM$ does not appear in the small divisors, then we proceed as in case 1.\\
	
	If $\al_{h_{\tM}}+\bt_{h_{\tM}}\ge 2$ then we have
	\begin{equation}\label{caffe}
	\pa{2^{{h_{\tM}}(\al_{h_{\tM}} + \bt_{h_{\tM}} -1)}}^{-\frac{\delta}{2}} \pa{1 + h_{\tM}^2 \pa{\al_{h_{\tM}} + \bt_{h_{\tM}}}^2}^\tau \le C(\delta) \lesssim {\frac{1}{\delta^2} (\ln(1/\delta))^2}.
	\end{equation}
	Indeed, letting $x= \al_{h_{\tM}} + \bt_{h_{\tM}} -1 \ge 1$, the left hand side is bounded by 
	$$ \exp\pa{-\bar\delta h_{\tM} x + \log\pa{1 + h_{\tM}^2(x + 1)^2}} \le \exp\pa{-\bar\delta h_{\tM} x + \log\pa{1 + 4 h_{\tM}^2 x^2}} \le 4 \exp{\pa{-\bar\delta h_{\tM} x + \log\pa{1 + h_{\tM}^2 x^2}} }. $$ Now, if $h_{\tM} x \ge \frac{4}{\delta}\log\delta^{-1} =: \bar{y}_{\delta}$, then the right hand side is bounded by $4$. Otherwise, the bound is given by $4(1 + \bar{y}_\delta^2)$.
	Then in this case the right hand side of \eqref{qulo} is bounded by
	\[
	C(\delta) \exp{\pa{-\frac{\delta \log 2}{2} \sum_h \jap{h}(\al_h + \bt_h) + \tau \sum_h \log\pa{1 + \jap{h}^2(\al_h + \bt_h)^2}}}
	\]
	and we proceed as in case 1.
	
	\medskip
	\noindent
	\gr{(b)} If $\al_{h_{\tM}}+\bt_{h_{\tM}}=1$, here the second factor in \eqref{qulo} is equal to $1$. Thus in order to bound the third factor (i.e.
	$(1+ h_\tm^2)$)  we need a different argument.
	By\footnote{note that since $\al_{h_{\tM}}+\bt_{h_{\tM}}=1$, then $\al_{h_{\tM}}+\bt_{h_{\tM}}= |\al_{h_{\tM}}-\bt_{h_{\tM}}|$ and then $2^{h_\tM}\le |m_1(\fa-\fb)|$, in the notations of  \cite{BMP:2019}, and we may apply the Lemma.} Lemma C.4 of \cite{BMP:2019}, we have
	\begin{equation}\label{mars}
	\na_2=2^{h_\tM} \le 31 \sum_{l\ge 3} \na_l^2 = 31 \sum_{h<h_\tM} 4^{\jap{h} }(\al_h+\beta_h)\,.
	\end{equation}
	Then, the right hand side of \eqref{qulo} is bounded by\footnote{ Since $h_\tM\ge 1$ we can bound $1+ h_\tm^2\le 2 h_\tm^2$.}
	\begin{equation}\label{franco}
	2^\tau \;\sup_{ \fa, \fb } \pa{\prod_{h< h_{\tM}} 2^{\jap{h}(\al_h + \bt_h)}}^{-\frac{\delta}{2}} 
	\pa{\log\big(31\sum_{h<h_\tM} 4^{\jap{h} }(\al_h+\beta_h)\big)}^{2\tau}\prod_{h<h_{\tM}} \pa{1 + \jap{h}^2 \abs{\al_h + \bt_h}^2}^\tau \g^{-1} .
	\end{equation}
	Now, since $4^{\jap{h}} (\al_h+\beta_h)\ge 2$ we have
	\begin{align}\label{sgamuffo}
	\begin{aligned}
	\log(\sum_{h<h_\tM} 4^{\jap{h} }(\al_h+\beta_h)) \le  \log(\prod_{h<h_\tM} 4^{\jap{h} }(\al_h+\beta_h)) &\le \sum_{h < h_\tM} \log_2 (4^{\jap{h} (\al_h + \bt_h)}) \le \\ \sum_{h < h_\tM} 2\jap{h}(\al_h + \bt_h) \le \prod_{h<h_{\tM}} \pa{1 + \jap{h}^2 (\al_h+\beta_h)^2}.
	\end{aligned}
	\end{align}
	In conclusion, \eqref{franco} is bounded by 
	\begin{equation}
	\label{marcello}
	2^{3\tau} \exp{\pa{-\frac{\delta \log 2}{2} \sum_h \jap{h}(\al_h + \bt_h) + 3\tau \sum_h \log\pa{1 + \jap{h}^2(\al_h + 
				\bt_h)^2}}} \stackrel{\eqref{bibi}}{\le} 2^{3\tau} \exp\pa{3\tau \frac{24}{\bar\delta} \log^2{\frac{1}{\bar\delta}}}\,.
	\end{equation}
		\gr{ Case 3. $\na_1 > 2^{h_\tM} = \na_2$ and two normal sites.}
Noticing that $\na_1\notin \cS$,  we denote by $j_2$ the second normal site.	Now
	\begin{equation}\label{tre cappucci}
	\Pi_{l\ge 3}  \na_l =  \jjap{j_2}2^{h_\tM(\al_{h_\tM} + \bt_{h_\tM}-1)} \prod_{h < h_\tM} 2^{\jap{h}(\al_h + \bt_h)}\,.
	\end{equation}
	The only difference is that now we need to bound
	\begin{equation}\label{marocchino}
	\sup_{ \fa, \fb } \pa{\jjap{j_2}\prod_{h< h_{\tM}} 2^{h(\al_h + \bt_h)}}^{-\frac{\delta}{2}} 
	\pa{2^{h_{\tM}(\al_{h_{\tM}} + \bt_{h_{\tM}} -1)}}^{-\frac{\delta}{2}} \pa{1 + h_{\tM}^2 \abs{\al_{h_{\tM}} - \bt_{h_{\tM}}}^2}^\tau\prod_{h<h_{\tM}} \pa{1 + h^2 \abs{\al_h - \bt_h}^2}^\tau \g \,.
	\end{equation}
	\gr{(a)} If $\al_{h_{\tM}}=\bt_{h_{\tM}}$, or if $\al_{h_{\tM}}+\bt_{h_{\tM}}\ge 2$ then we proceed as in case 2-(a), since $\jjap{j_2}^{-\frac{\delta}{2}}\le 1$ and does not bother.\\
	\gr{(b)} If $ \al_{h_{\tM}}+\bt_{h_{\tM}}=1$, then \eqref{marocchino} is bounded by
	\begin{equation}\label{con sambuca}
	2^\tau \;\sup_{ \fa, \fb } \pa{\jjap{j_2}\prod_{h< h_{\tM}} 2^{\jap{h}(\al_h + \bt_h)}}^{-\frac{\delta}{2}} 
	\pa{\log\big(31\sum_{h<h_\tM} 4^{\jap{h} }(\al_h+\beta_h) +31 \jjap{j_2}^2\big)}^{2\tau}\prod_{h<h_{\tM}} \pa{1 + \jap{h}^2 \abs{\al_h - \bt_h}^2}^\tau \g^{-1} .
	\end{equation}
	We have that\footnote{we use the fact that $\log(x + y) \le \log x + \log y$, if $x,y \ge 2$ } 
	\begin{align}\label{mosca}
	\log\big(31\sum_{h<h_\tM} 4^{\jap{h} }(\al_h+\beta_h) +31 \jjap{j_2}^2\big) &\le \log\pa{31\sum_{h<h_\tM} 4^{\jap{h} }(\al_h+\beta_h) } + \log\pa{\jjap{j_2}^2} \\ &\le C \log\jjap{j_2} + \prod_{h\le h_\tM}\pa{1 + \jap{h}^2(\al_h + \bt_h)^2} \,,
	\end{align}
	where we proceeded as in case 2-(b) like in \eqref{sgamuffo}. 
	In \eqref{marocchino} all the terms of "tangential nature", are bounded in the same way  as in \eqref{marcello}.
	Concerning the ones depending on $\jjap{j_2}$, it suffices to note that
	\begin{equation}\label{bounty}
	-\bar{\delta} \log\jjap{j_2} + \log\log\jjap{j_2} \le 0,\qquad \text{whenever}\qquad \log\jjap{j_2} \ge \frac{2}{\bar\delta}\log\frac{1}{\bar\delta},
	\end{equation}
	while 
	\begin{equation}
	\label{sneakers}
	-\bar{\delta} \log\jjap{j_2} + \log\log\jjap{j_2} \le 2\log\frac{1}{\bar\delta}\qquad \text{otherwise},
	\end{equation}
	with $\bar\delta = \frac{\delta}{4\tau C}$.
	\subsection*{Case 4. $\na_1 = 2^{h_\tM}$ and the (eventual) normal sites are $< \na_2$}
	Let us start with the case of two normal sites which we denote by $j_1\ge j_2$.
	Here the expression \eqref{tre cappucci} becomes 
	\begin{equation}\label{tre cappucci normali}
	\Pi_{l\ge 3}  \na_l =  \jjap{j_1}\jjap{j_2}2^{h_\tM(\al_{h_\tM} + \bt_{h_\tM}-1)} \prod_{h < h_\tM} 2^{\jap{h}(\al_h + \bt_h)}\,,
	\end{equation}
	which contributes with the additional factor $\jjap{j_1}^{-\frac{\delta}{2}}$ in the expression \eqref{marocchino}.\\
	\gr{(a)} If $\al_{h_{\tM}}=\bt_{h_{\tM}}$, or if $\al_{h_{\tM}}+\bt_{h_{\tM}}\ge 2$ then we are reduced to case 3-(a).\\
	\gr{(b)} If $\al_{h_{\tM}}+\bt_{h_{\tM}} = 1$ then, we proceed as in case 2-(b) and apply Lemma C.4\footnote{in the notation of \cite{BMP:2019} $\na_1 = |m_1|$}, following the same estimates as in \eqref{mosca}
	\begin{align}\label{twix}
	\begin{aligned}
	\pa{1 + h_{\tM}^2 \abs{\al_{h_{\tM}} - \bt_{h_{\tM}}}^2} = \pa{1 + h_{\tM}^2} &\le 2 \pa{c\log\big(31\sum_{h<h_\tM} 4^{\jap{h} }(\al_h+\beta_h) +31 \jjap{j_1}^2 + 31 \jjap{j_2}^2\big)} \\ 
	&\le c \log\jjap{j_1}+ \prod_{h\le h_\tM}\pa{1 + \jap{h}^2(\al_h + \bt_h)^2}
	\end{aligned}
	\end{align}
	and we conclude as in case 3-(b), with \eqref{bounty} applied twice.
	
	If there is only one normal site or if there is none, then the same arguments apply word by word with the only "advantage" that there is only one $\jjap{j_1}$ or none in \eqref{tre cappucci normali}
	\subsection*{Case 5. $\na_1= 2^{h_{\tM}}$ and only one normal site $j_1=\na_2$} Here
	\[
	\Pi_{l\ge 3}  \na_l =  2^{h_\tM(\al_{h_\tM} + \bt_{h_\tM}-1)} \prod_{h < h_\tM} 2^{\jap{h}(\al_h + \bt_h)}\,.
	\]
	\gr{(a)}  If $\al_{h_{\tM}}=\bt_{h_{\tM}}$, or if $\al_{h_{\tM}}+\bt_{h_{\tM}}\ge 2$ then we proceed as in case 2-(a)\\
	\gr{(b)} If $\al_{h_\tM} + \bt_{h_\tM} =1$, then we are like in case 2-(b) and apply Lemma C.4 to $\na_1 = m_1$ obtaining the same bound as in \eqref{mars}.
	\subsection*{Case 6. $\na_1= 2^{h_{\tM}}$, $j_1=\na_2$ and two normal sites }
	The proof follows word by word the one of Case 3.\\

\end{proof}

\appendix
\section{Iterative Lemma and Proof of Theorem \ref{allaMoserbis}}\label{AAA}
 Fix
 $r_0, \rho, \delta$ as in \eqref{newton}
  and let $\{\rho_n\}_{n\in\N}, \{\delta_n\}_{n\in\N}$ be  the  summable  sequences:
 	\begin{equation}\label{amaroni}
 	\rho_n= \frac{\rho}{{6}} 2^{-n}\,,\qquad \delta_0 = \frac{\delta}{8}, \quad \delta_n = {\frac{9\delta}{4\pi^2 n^2}}\quad \forall n\ge 1\,.
 \end{equation} 
   Let us define recursively
   \begin{eqnarray}\label{pestob}
 &&r_{n+1} = r_n - {3}\rho_n\ \to \ r_\infty:=r_0-\rho\qquad   {\rm (decr
 easing)} \nonumber\\
 &&p_{n+1} = p_n + {3}\delta_n\ \to \ s_\infty:=
 p_0+\delta\qquad  {\rm (increasing)} \nonumber \\
\end{eqnarray}
Note that for every $ r'\geq r_\infty\,, p'\leq p_\infty$
\begin{equation}\label{cumdederit}
\sqrt{I}\in {\bar B}_{r}(\tw_{p_\infty})
\qquad
\stackrel{\eqref{newton}}\Longrightarrow\qquad
\sqrt{I}\in 
{\bar B}_{\frac{r_0}{2\sqrt 2}}(\tw_{p_\infty})
\subset
{\bar B}_{\frac{r_\infty}{\sqrt 2}}(\tw_{p_\infty})
\subset {\bar B}_{\frac{r'}{\sqrt 2}}(\tw_{p'})
\end{equation}
and that the projections $\Pi^d$ are well defined  
on every space
$\cH_{r',p'}$.
{We define  $\cH_{r,p}^{0,\cK}\subset \cH_{r,p}$to  be the subspace of \emph{counter-terms}, i.e. Hamiltonians of the form
	\begin{equation}\label{silvestre}
	\Lambda = \sum_{j\in \cS} \lambda_j( |v_j|^2 -I_j)	 + \sum_{j\in\cS^c} \lambda_j |z_j|^2\,, \quad \lambda:=\pa{\lambda_j}\in \ell_\infty :\quad \norma{\Lambda}_{r,p} \equiv  \norma{\lambda}_\infty.
	\end{equation}
	This space $\cH_{r,s}^{0,\cK}$
	can be isometrically identified with $\ell^\infty,$ namely
	\begin{equation}
	\label{grevity}
	\Lambda \in\ell^\infty\ \ \mbox{and}\ \  
	\| \Lambda\|_\infty:=\|\lambda\|_\infty\,.
	\end{equation} 
	For the proof of these facts, see Lemma 4.3 of \cite{BMP:almost}.}

In conclusion (see Lemma 4.4 of \cite{BMP:almost})  for any $H\in \cH_{r',p'}$ with
$r'\geq r_\infty\,,\  p'\leq p_\infty$, one has 
\begin{equation}\label{ritornello}
\|\Pi^0 H\|_{r,p}\,,\ \| \Pi^{0,\cK} H\|_\infty \leq  3  \|H\|_{r,p}
\,, \quad
\|\Pi^{(-1)} H\|_{r,p},\|\Pi^{(-2)} H\|_{r,p}  \leq   \|H\|_{r,p}\,,\quad
\|\Pi^{\geq 2} H\|_{r,p}  \leq  5  \|H\|_{r,p}\,.
\end{equation}

Let  
\begin{equation}\label{bisanzio}
H_0 := D(\o) + G_0 + \Lambda_0\,,
\qquad
G_0\in \cH_{r_0,p_0}\,,\qquad
\Lambda_0\in\ell^\infty\,,
\end{equation}
(recall \eqref{grevity})
where 
the counterterms $\Lambda_0$ 
are free parameters. 
We define 
\begin{equation}\label{vigili}
\e_0:=\gamma^{-1}\pa{\norma{\zeroK{G_0}}_\infty + \norma{\zeroR{G_0}}_{r_0, p_0} + \norma{\due{G_0}}_{r_0, p_0} + {\norma{G_0^{(-1)}}_{r_0, p_0}}} ,  \quad \Theta_0:= \gamma^{-1}\norma{\buon{G_0}}_{r_0, p_0} +\e_0 
\end{equation}

\begin{lemma}[Iterative step]\label{iterativo}
Let\footnote{Recall also that $0<\theta<1$ was fixed one and for all.} $r, r_0,p_0, \rho, \delta$ be  as in \eqref{newton},
  $\rho_n, r_n, p_n, $  as in \eqref{amaroni}-\eqref{pestob},
  $H_0,G_0,\Lambda_0$ as in \eqref{bisanzio}
  and
$\eps_0,\Theta_0$  as in \eqref{vigili}.
Let $\sqrt{I}\in {\bar B}_{r}(\tw_{p_\infty})$.
 There exists  a constant $\frak C>1$ large enough
 such that
if 
\begin{equation}\label{gianna}
\eps_0 	\leq \pa{1 + \Theta_0}^{-5} \tK^{-2}\,,
\qquad 
\tK
:= \frak C  \pa{\frac{r_0}{\rho}}^6\sup_n 2^{6n} e^{\crac n^3} e^{-\chi^n (2-\chi) } \,,
\qquad
\crac := 
\frac{2^7\, c}{\delta^{3/2}} \,,
\end{equation}
{($c$ is defined in Lemma \ref{mah})}
then we can iteratively construct a sequence of generating functions 
$S_i = \due{S_i} + S_i^{(-1)}+ \zero{S_i}\in\cH_{r_{i}-\rho_i, p_{i+1}}$ 
and a sequence of  counterterms  $\bar\Lambda_i\in\ell^\infty$ such that the following holds, for $n\ge 0$.

$(1)$ For all $ i = 0,\ldots,  n -1 $ and any $ p'\ge p_{i+1}$ the 
time-1 hamiltonian flow
 $\Phi_{S_i}$ generated by $S_i$   satisfies
	\begin{equation}
\sup_{u\in  {\bar B}_{r_{i+1}}(\tw_{p'})} \norm{\Phi_{S_i}(u)- u}_{p'} \le \rho 2^{-2i-7} \,.\label{ln}
 \end{equation}
 Moreover
	\begin{equation}\label{ucazzo}
	\Psi_n := \Phi_{S_0}\circ\cdots \circ \Phi_{S_{n-1}} 
	\end{equation}
	is a well defined, analytic map ${\bar B}_{r_n }(\tw_{p'}) \to {\bar B}_{r_0}(\tw_{p'})$ for all $p'\ge p_n $ with the bound
	\begin{equation}
	\label{cosi}
	 \sup_{u\in {\bar B}_{r_n}(\tw_{p'})}\abs{\Psi_{n}(u) - \Psi_{n-1}(u)}  \le  \rho 2^{-2n + 2}.
	\end{equation}

	$(2)$ We set $\cL_0:=0$  and for $i=1,\dots,n$ 
	 $$ \mathcal{L}_{i} + \id := e^{\set{S_{i-1},\cdot}}\pa{\mathcal{L}_{i-1} + \id},\quad \Lambda_{i} := \Lambda_{i-1} - \bar{\Lambda}_{i-1}\,,\quad H_i= e^{\{S_{i-1},\cdot\}}H_{i-1}$$ where $\Lambda_{i-1}$ are free parameters and $\mathcal{L}_{i} : \ell^\infty\to\cH_{r_{i}, p_{i}}$ 
are linear operators.  
	 We have
	\begin{equation}
\label{cioccolato}
H_{i} = D(\o) + G_{i} + \pa{\id + \mathcal{L}_{i}}\Lambda_{i},\qquad G_{i}, \in\cH_{r_{i}, p_{i}}.
\end{equation}
 Setting for $ i = 0,\ldots, n  $
	\begin{equation}\label{xhx-i}
	 \eps_i:=\gamma^{-1}\pa{\norma{\zeroK{G_i}}_\infty + \norma{\zeroR{G_i}}_{r_i, p_i} + \norma{\due{G_i}}_{r_i, p_i} + {\norma{G_i^{(-1)}}_{r_i, p_i}}},  \quad \Theta_i:=\gamma^{-1} {\norma{G_i^{\ge 1}}_{r_i,p_i}} +\e_i \,,
	\end{equation}
we have
\begin{eqnarray}
& \e_i \leq   \e_0  e^{- \chi^{i}+1} \,, 
\qquad
\chi:=3/2\,,\qquad
\qquad  \Theta_i \leq   \Theta_0 \sum_{j=0}^i 2^{-j}\, \label{en} \\
& 
 \label{fringe}
\norma{\pa{\call_{i} -\call_{i-1}} h}_{r_i,p_i}\le \tK \eps_0 \pa{1 + \Theta_0}^2 2^{-i} \norma{h}_\infty,\qquad  \norma{\call_i h}_{r_i,p_i} \le \tK (1+  \Theta_0)^2\e_0\sum_{j=1}^i 2^{-j}\norma{h}_\infty,
\end{eqnarray}
for all $h\in\ell^\infty.$
Finally the counter-terms satisfy the bound
 \begin{equation}
 \norma{\bar{\Lambda}_{i-1}}_\infty \le  \g \tK \e_{i-1}(1+\Theta_0)^{{2}}\,,\quad i=1,\dots,n \,. \label{lambdonebarra iterstima}
	\end{equation}
\end{lemma}

\begin{proof}[Proof of Theorem \ref{allaMoserbis}] Starting from the Hamiltonian $H$ satisfying \eqref{maipendanti}, we set $G_0 = H - D(\o)$ in \eqref{bisanzio}. The smallness conditions \eqref{gianna} are met, provided that we choose $\bar\epsilon$ and $\bar C$ appropriately. \\
Using \eqref{cosi} we define $\Psi$ as the limit of the $\Psi_n$ (which define a Cauchy sequence) and $\Lambda = \Lambda_0 = \sum_j \bar\Lambda_j < \infty$. Note that the series is summable by \eqref{lambdonebarra iterstima}. For more details see \cite[Section $6$]{BMP:almost}.
\end{proof}

\begin{proof}[Proof of the iterative Lemma]
We start with a Hamiltonian
	$H_0= D_\betta + \Lambda_0 +G_0$
with $\Lambda_0\in \ell^\infty$. 

At the $n$'th step we have an expression of the form
\[
H_n=D_{\betta} +\pa{\id+\call_n}\Lambda_n+G_n
\]
with $G_n\in \cH_{r_n,p_n}$,
\\
To proceed to the step $n+1$ we apply the change of variables $e^{\{S_n,\cdot\}}$. The generating function $S_n$ and the counterterm $\bar\Lambda_n$ are fixed as the unique solutions of the Homological equation
\begin{equation}\label{homo sapiens}
\Pi^{\le 0}  \pa{\set{S_n,D_{\betta}+G_n^{\ge 1}} +\pa{\id+\call_n}\bar\Lambda_n+G_n }= G_{n}^{(-2, \cK)}\,.
\end{equation}
This equation can be written componentwise as a triangular system and solved consequently. We have
\begin{align}
	& {\set{S_n^{(-2)},D_{\betta}} +\Pi^{-2,\cR}\call_n\bar\Lambda_n+G^{(-2,\cR)}_n }=0\\
	&\set{S_n^{(-1)},D_{\betta}}+\Pi^{-1}\set{S_n^{(-2)},G_n^{\ge 1}} +\Pi^{-1}\call_n\bar\Lambda_n+G^{(-1)}_n =0\\
	\label{counter}
	&\Pi^{0,\cK}\set{S_n^{(-2)}+ S_n^{(-1)},G_n^{\ge 1}} +\bar{\Lambda}_n+\Pi^{0,\cK}\call_n\bar\Lambda_n+G^{(0,\cK)}_n=0\\
	&\set{S_n^{(0,\cR)},D_{\betta}}+\Pi^{0,\cR}\set{S_n^{(-2)}+ S_n^{(-1)},G_n^{\ge 1}} +\Pi^{0,\cR}\call_n\bar\Lambda_n+G^{(0,\cR)}_n = 0\,.
\end{align}
We start by solving the equations for $S_n$ it "modulo $\bar\Lambda_n$", then we determine the counter-term by inversion of an appropriate linear operator resulting from inserting the equations for $S_n$ into equation \eqref{counter}. \\
We hence have
\begin{align}
	& S_n^{(-2)}= L_{\betta}^{-1} \pa{\Pi^{-2}\call_n\bar\Lambda_n+G^{(-2)}_n }\label{ostrica}\\
	&S_n^{(-1)}= L_{\betta}^{-1}\pa{\Pi^{-1}\set{L_{\betta}^{-1} \pa{\Pi^{-2}\call_n\bar\Lambda_n+G^{(-2)}_n },G_n^{\ge 1}} +\Pi^{-1}\call_n\bar\Lambda_n+G^{(-1)}_n}\nonumber \\
	& S_n^{(0,\cR)} = L_{\betta}^{-1} \pa{\Pi^{0,\cR}\set{S_n^{(-2)}+ S_n^{(-1)},G_n^{\ge 1}} +\Pi^{0,\cR}\call_n\bar\Lambda_n+G^{(0,\cR)}_n }\,.\nonumber
\end{align}
Plugging them into \eqref{counter} we thus get
\begin{align*}
	&\Pi^{0,\cK}\set{L_{\betta}^{-1} \pa{\Pi^{-2}\call_n\bar\Lambda_n  +\Pi^{-1}\set{L_{\betta}^{-1} {\Pi^{-2}\call_n\bar\Lambda_n },G_n^{\ge 1}}},G_n^{\ge 1}} +\bar{\Lambda}_n+\Pi^{0,\cK}\call_n\bar\Lambda_n=\\
	&\Pi^{0,\cK}\set{L_{\betta}^{-1} \pa{G^{(-2)}_n  +\Pi^{-1}\set{L_{\betta}^{-1} {G^{(-2)}_n },G_n^{\ge 1}}},G_n^{\ge 1}}\,.
\end{align*}
Since $M_n: \ell^\infty\to \ell^\infty$ defined as
\[
M_n h = \Pi^{0,\cK}\set{L_{\betta}^{-1} \pa{\Pi^{-2}\call_n h  +\Pi^{-1}\set{L_{\betta}^{-1} {\Pi^{-2}\call_n h },G_n^{\ge 1}}},G_n^{\ge 1}} +\Pi^{0,\cK}\call_n h
\]
satisfies 
$
\norma{M_n h}_{\infty} \le \norma{ h}_{\infty}/2
$,
then 
 \[
\bar{\Lambda}_n =(\id +M_n)^{-1}(\Pi^{0,\cK}\set{L_{\betta}^{-1} \pa{G^{(-2)}_n  +\Pi^{-1}\set{L_{\betta}^{-1} {G^{(-2)}_n },G_n^{\ge 1}}},G_n^{\ge 1}})
\]
 is determined and satisfies\footnote{The crucial point is that no small divisor appears in this estimate. This is due to the property $\norma{\Lambda}_{r,p} \equiv  \norma{\lambda}_\infty$ for all $r,p$. The detailed computation can be found in formula (6.32) of \cite{BMP:almost}}
 
 \begin{equation}
 \label{lambdozzo}
 \norm{\bar\Lambda_n}_{\infty} \lesssim \g  (1 + \Theta_0)^2\eps_n.
 \end{equation}
 
By substituting  in the equations \eqref{ostrica}, we get the final expressions for $S^{(-2)}_{n}$ and $S^{(-1)}_n$ and finally $S^{(0,\cR)}_n$ which yields the estimates

\begin{align}
\norm{\due{S_n}}_{r_n, p_n + \delta_n} & \lesssim  D_n \eps_n\\ 
\norm{S_n^{(-1)}}_{r_n - \rho_n, p_n + 2\delta_n}  &\lesssim  \pa{\frac{r_0}{\rho_n}} D_n^2 \eps_n (1 + \Theta_0)\\
\norm{S_n^{(0)}}_{r_n - 2\rho_n, p_n + 3\delta_n} & \lesssim  \pa{\frac{r_0}{\rho_n}}^2 D^3_n (1 + \Theta_0)^2 \eps_n,
\end{align}
 where we defined $$D_n :=  e^{\frac{c}{\delta_n} \ln^2(\frac1{\delta_n})} $$ as the constant coming from the homological equation and systematically used the inductive hypothesis and the first bound  in \eqref{gianna}.
 The final bound thus reads
 
 \begin{equation}
 \label{pasta frolla}
 \norm{S_n}_{r_n - 2\rho_n, p_n + 3\delta_n} \lesssim   \pa{\frac{r_0}{\rho_n}}^2  e^{{\frac{3c}{\delta_n}} \ln^2(\frac1{\delta_n})}\eps_n (1 + \Theta_0)^2.
 \end{equation}
 Then following word by word the corresponding computation in the proof of \cite[Lemma 6.1]{BMP:almost} we prove item (1) of Lemma \ref{iterativo}.\\ Regarding item (2), by construction we have
 $$
 \call_{n+1} - \call_{n} = \pa{e^\set{S_n,\cdot} - \id} \circ (\call_n + \id)
 $$
 hence 
 $$
 \norma{ (\call_{n+1} - \call_{n})h }_{r_{n+1}, p_{n+1}} \le \pa{\frac{r_0}{\rho_n}}^3  e^{{\frac{3c}{\delta_n}} \ln^2(\frac1{\delta_n})}\eps_n (1 + \Theta_0)^2 \norma{h}_\infty,
 $$
 which a fortiori proves \eqref{fringe}.\\
As for the expression of $G_{n+1}$, by definition we have 
\[G_{n+1} = e^{\set{S_n,\cdot}} H_n - \sq{D(\o) + \pa{\id + \call_{n+1}}\Lambda_{n+1}}\,.
\]
Since $S_n$ solves the Homological equation \eqref{homo sapiens}, 
we have that
\begin{align}\label{schifo al cazzo}
G_{n+1} &= \dueK{G_n} + G_n^{\ge 1}+\Pi^{\ge 1}\pa{\call_{n+1} \bar{\Lambda}_n+\set{S_n,G_n^{\ge 1}}}  + G_{n+1,*}
	\\ G_{n+1,*} &=  \{S_n,G_n^{\le 0}\} + \Pi^{\le 0}\pa{\call_{n+1}-\call_n}\bar{\Lambda}_n+
	\pa{e^{\{S_n,\cdot\}} - \id - \set{S_n,\cdot}}
	G_n  \nonumber \\
	& - \sum_{h=2}^\infty \frac{ (\ad S_n)^{h-1}}{h!} \pa{\Pi^{\le 0}\pa{\id + \call_n}\bar\Lambda_n + G_n^{\le 0} +\Pi^{\le 0} \{S_n^{(-1)} + \due{S_n},G_n^{\ge 1}\} - G_n^{(-2,\cK)}} \nonumber.
\end{align}
Note  that $G_{n+1,*}$ is  quadratic in $S_n \sim G_n^{\le 0}$.\\
We finally have
\begin{equation}\label{stellina}
\g^{-1}\norma{G_{n+1,*}}_{r_{n+1}, p_{n+1}} \lesssim \tK 2^{6n} e^{\crac n^3} \pa{1 + \Theta_0}^5 \eps_n^2 \le \tK^2(1 + \Theta_0)^5\eps_0^2 \, e^{-\chi^{n+1}} \stackrel{\eqref{gianna}}{\le} \eps_0 e^{-\chi^{n+1}},
\end{equation}
by the definition of $\tK$ and $\delta_n$ (of course the constant $\crac$ is such that  $ {{\frac{6c}{\delta_n}}  \ln^2(\frac1{\delta_n})} \le \crac n^3,\, \forall n\ge 0$). This implies the first estimates in \eqref{en}.\\
 Similar calculations apply to those terms in \eqref{schifo al cazzo} of degree $\ge 1$.
More precisely we obtain 
$$
\Theta_{n+1} \le \Theta_n(1 + \pa{\frac{r_0}{\rho_n}}^2  e^{{\frac{3c}{\delta_n}} \ln^2(\frac1{\delta_n})}\eps_n (1 + \Theta_0)^2 )
$$
which implies the second estimates in \eqref{en} by the definition of $\tK$ and the smallness condition \eqref{gianna}.
\end{proof}
\section{Measure estimates}\label{measure}
Let us show that $\cC$ has a positive measure in $\cQ_{\cS}$.
We start with the following
\begin{lemma}\label{carciofo}
	If $\ell\in\Z^\Z$ satisfies $\fm(\ell)=\pi(\ell)=0,$ then
	$|\ell|$ is even and $|\ell|\neq 2.$
\end{lemma}
\begin{proof}
	As usual we write in a unique way $\ell=\ell^+ -\ell^-$ where
	$\ell^\pm_s\geq 0$  and $\ell^+_s\ell^-_s= 0$ for every $s\in\Z.$
	Since  $\fm(\ell)=0$ we get $|\ell^+|= |\ell^-|$, therefore
	$|\ell|=|\ell^+|+ |\ell^-|$ is even.\\
	Now assume by contradiction that $|\ell^+|= |\ell^-|=1.$ Then
	$\ell^+=e_i,\ell^-=e_j$ for some $i\neq j$ and $\pi(\ell)=i-j,$
	which contradicts $\pi(\ell)=0.$
\end{proof}
\begin{proof}[Poof of Lemma \ref{misurino}] 
	Take $\gamma\leq 1/2.$
For $\ell\ne 0$ with $|\ell|<\infty$, $\sum_{j\in \cS^c} |\ell_j|\le 2$, $\pi(\ell)=0$ and $\fm(\ell)=0$
we define the 
resonant set 
\[
\cR_\ell:=\set{\bo\in \cQ_{\cS}:	|\betta(\bo)\cdot \ell|\le \gamma \prod_{j\in \cS}\frac{1}{(1+|\ell_j|^2\jap{\log_2 j}^{2})^\tau}}
\]
	Set $\tk(\ell):= \sum_{s\in \Z} s^2 \ell_s.$
	If $|\tk(\ell)|\geq |\ell|$ 
	then 
	\[
	|\omega\cdot \ell|\ge | \sum_{s\in \Z} s^2 \ell_s|  - \frac12|\ell|\ge\frac12 |\ell|\ge \frac12
	\]
	and 	 $\cR_{\ell}=\emptyset.$ 
	Then
	\[
	\meas(\cQ_\cS\setminus\cC) \le \sum_{\substack{\ell: 0<|\ell|<\infty,  \sum_{j\in \cS^c}  |\ell_j|\le 2 \\ \pi(\ell)=\fm(\ell)= 0 \\ |\tk(\ell)| <|\ell|} } \meas(\mathcal R_\ell)\,.
	\]
	We note that there exists $\bar s\in\cS$ such that
	$\ell_{\bar s}\neq 0,$ 
	otherwise $0<|\ell|=\sum_{s\in \cS^c}|\ell_s|\le 2$, which contradicts Lemma
	\ref{carciofo}.
Recalling $\omega= (\nu,\Omega(\nu))$, we get
	\[
	 t^{-1}|(\omega(\nu)\cdot\ell)- (\omega(\nu + t e_{\bar s})\cdot\ell) |\ge |\ell_{\bar s}|-2\sup_{j\in \cS^c}|\Omega_j|^{\rm lip}\ge 1-O(\epsilon)\geq 1/2
	\]
	and 
	$$
	\meas(\mathcal R_\ell) \le 
	\gamma \prod_{s\in \cS}\frac{1}{(1+|\ell_s|^2 \langle\log_2s\rangle^2)^\tau}
	=
	\gamma \prod_{i\in\N}\frac{1}{(1+|\ell_{2^i}|^2 \langle i\rangle^2)^\tau}
	\,.
	$$
	Therefore 
	\begin{equation}\label{istanbul}
	\meas(\cQ_\cS\setminus \cC) \le \sum_{\ell\in A}
	\gamma \prod_{i\in\N}\frac{1}{(1+|\ell_{2^i}|^2 \langle i\rangle^2)^\tau}
	\,,
	\end{equation}
	where
	$$
	A:=\left\{ \ell\in\Z^\Z\ :\   0<|\ell|<\infty,\   \sum_{j\in \cS^c}  |\ell_j|\le 2,\ 
	\fm(\ell)=\pi(\ell)= 0, \  |\tk(\ell)| <|\ell|
	\right\}\,.
	$$
	Given $k\in\Z^\N$ with $|k|<\infty$ we define $\ell^k\in\Z^\Z$
	supported on $\cS$ setting
	$$
	\ell_{s}:=k_{\log_2s}\,,\ \ {\rm for}\ \ s\in\cS\,,\ \  \ell_s=0\ \ {\rm for}\ \ s\notin\cS\,.
	$$
	Now each $\ell\in A$ there exists unique $k\in\Z^\N$ with $|k|<\infty$
	and
	$j_1,j_2\in \cS^c$ and $\s_1,\s_2=\pm 1,0$ such that
	\begin{equation}\label{apapaia}
	\ell = \ell^k+ \s_1 \be_{j_1} + \s_2 \be_{j_2}\,.
	\end{equation}
	On the other hand, given $k\in\Z^\N$ with $|k|<\infty$, 
	there exist at most $36(|k|+2)$ vectors $\ell\in A$ satisfying \eqref{apapaia}.
	Indeed we prove that, given  $\s_1,\s_2=\pm 1,0$, there exist at 
	most\footnote{Note that there are 9 possible choices of $\s_1,\s_2=\pm 1,0$.}
	$4(|k|+2)$ couples
	$(j_1,j_2)\in\cS^c\times\cS^c$ such that $\ell$ in \eqref{apapaia} 
	belongs to $A$.
	Indeed they have to satisfy
	\[
	\begin{cases}
	\s_1j_1+\s_2j_2=-\pi(\ell^k) \\ 
	\s_1j_1^2+\s_2j^2_2 = - \tk(\ell^k) + h\,, \qquad {\rm for \ some\ \ }|h|<|\ell| \leq |k|+2\,.
	\end{cases}	
	\]
	Then by \eqref{istanbul} we get
	$$
	\meas(\cQ_\cS\setminus\cC) \le 36\g	 \sum_{k\in \Z^\N:  0<|k|<\infty} 
	(|k|+2) \prod_{i\in \Z}\frac{1}{(1+|k_i|^2 \langle i\rangle^2)^\tau} \,.
	$$
	Since
	\[
	|k|+2 \le  2\prod_{i\in \Z}(1+|k_i|^2 \langle i\rangle^2)^{1/2}
	\]
	we get
	$$
	\meas(\cQ_\cS\setminus\cC) \le 72\g	 \sum_{k\in \Z^\N:  0<|k|<\infty} 
	\prod_{i\in \Z}\frac{1}{(1+|k_i|^2 \langle i\rangle^2)^{\tau-1/2} }\,,
	$$
	where the last sum converges (see \cite{Bourgain:2005} or Lemma 4.1 of \cite{BMP:2019}) provided that $\tau\ge 3/2$.
\end{proof}
\section{Technicalities}
\begin{lemma}[Lemma C.2 of \cite{BMP:2019}]\label{luchino}
	Let $0<a<1$ and $x_1\geq x_2\geq \ldots\geq x_N\geq 2.$ Then
	$$
	\frac{\sum_{1\leq\ell\leq N} x_\ell}{\prod_{1\leq\ell\leq N} x_\ell^a}
	\leq 
	x_1^{1-a}+\frac{2}{a x_1^a}\,.
	$$
\end{lemma}
\begin{lemma} \label{pesto} Let $0 < \delta < e^{-3}$ and $f: [1,\infty] \to \R,\, y\mapsto-\delta y + \log(1 + y^2)$.  If $y \ge 4\delta^{-1}\log\frac{1}{\delta}$ then $$f(y) \le 0\,.$$
\end{lemma}
\begin{proof}
	We have $f(y)\le -\delta y + 1 + \log({\delta^2 y^2}) + \log{\frac{1}{\delta^2}}$. Let now consider the auxiliary function $$F(z) = -\frac{z}{2}  + \log z^2 + 1,$$ it has a maximum in $z_M=4$ then it is monotone decreasing, so there exists $z_0 $ such that $\forall z \ge z_0$  we have $F(z) \le 0$. One can take $z_0 = 12$. \\
	Then $$f(y) \le F(\delta y) - \frac{\delta y}{2}  + \log{\frac{1}{\delta^2}} \le  - \frac{\delta y}{2} + \log \frac{1}{\delta^2},$$ if $\delta y \ge 12$. Then, if $y \ge \frac{4}{\delta} \log \frac{1}{\delta}$, we got the thesis.
\end{proof}

\bibliographystyle{plain}
\bibliography{biblioAlmost2h}
\end{document}